\documentclass{article}
\usepackage{graphicx,amssymb,amsmath,amsthm,geometry,xcolor,mathrsfs,todonotes,hyperref,tikz} 

\newtheorem{definition}{Definition}
\newtheorem{theorem}{Theorem}

\newtheorem{lemma}{Lemma}
\newtheorem{corollary}{Corollary}

\newtheorem{claim}{Claim}[theorem]
\newenvironment{claimproof}[1]{\par\noindent\underline{Proof of claim:}\space#1}{\hfill $\blacksquare$}

\DeclareMathOperator{\Csp}{CSP}

\DeclareMathOperator{\Sg}{Sg}
\DeclareMathOperator{\Cg}{Cg}
\DeclareMathOperator{\Eq}{Eq}
\DeclareMathOperator{\Con}{Con}

\DeclareMathOperator{\Var}{\mathcal{V}}
\DeclareMathOperator{\A}{\mathbb{A}}

\tikzset{myStyle/.style={baseline=(center.base), font=\small,
    every node/.style={inner sep=0.25em} }}

\newcommand{\SquareUnwrapped}[4]{ 
  \node (center) at (0.5,-0.5) {\phantom{$\cdot$}}; 
  \path (0,0)  node (nw) {$#2$}
      ++(1,0)  node (ne) {$#4$}
      ++(0,-1) node (se) {$#3$}
      ++(-1,0) node (sw) {$#1$};
  \draw (nw) -- (ne) -- (se) -- (sw) -- (nw);
}  
\NewDocumentCommand{\SquareXY}{ O{} O{} O{} O{} O{1} O{1} }{ 
  \begin{tikzpicture}[myStyle, xscale=#5*1, yscale=#6*1 ]
    \SquareUnwrapped{#1}{#2}{#3}{#4}
  \end{tikzpicture}
}  
\NewDocumentCommand{\Square}{ O{} O{} O{} O{} O{1} }{ 
  \SquareXY[#1][#2][#3][#4][#5][#5]
}

\title{The class of congruence meet semidistributive varieties is not strong Maltsev }
\author{Andrew Moorhead\thanks{Institut f\"ur Algebra, TU Dresden. Email: andrew\textunderscore paul.moorhead@tu-dresden.de. The author has received funding from the ERC (Grant Agreement no. 101071674, POCOCOP). Views and opinions expressed are however
those of the author only and do not necessarily reflect those of the European Union or the European Research
Council Executive Agency. The author thanks Keith Kearnes, Michael Kompatscher, and Matthew Moore, each for past stimulating discussions about free algebras, equational logic, commutators, and higher dimensional congruences. }}

\date{}

\begin{document}

\maketitle

\begin{abstract}
   We present a proof that there is no single finite package of identities which characterizes the class of congruence meet semidistributive varieties. 
\end{abstract}

\section{Introduction}


A variety $\Var$ of algebras is \emph{congruence meet semidistributive} if each congruence lattice of its members satisfies the implication
\[
\gamma \wedge \alpha = \gamma \wedge \beta \implies \gamma \wedge (\alpha \vee \beta) = \gamma \wedge \alpha.
\]
It is one among the several classical congruence conditions (e.g.\ distributivitiy, modularity,  $n$-permutability) that have proved most useful in the taxonomy of varieties. Congruence meet semidistributivity has been more stubborn that its counterparts towards a characterization by what is called a \emph{Maltsev condition}, which roughly speaking is a (possibly infinite) disjunction of a sequence of finitely presented packages of identities where each condition is weaker than its predecessor. For example, Maltsev conditions for congruence distributivity, modularity, and $n$-permutability were known by the early 1970's (see \cite{jonsson-cd}, \cite{DayTerms}, and \cite{HagemannMitschke}, respectively), but the first progress for congruence semidistributivity was not reported until the 1980's, when Cz{\'e}dli found a \emph{weak} Maltsev condition for congruence meet semidistributivity (see \cite{CzedliCharSDMeet}). In the late 1990's Kearnes and Szendrei (see \cite{KearnesSzendreiRelTwoCommutators}) and Lipparini (see \cite{LippariniNeutral}) showed that congruence meet semidistributivity is truly characterized by a Maltsev condition. Shortly afterwords, Willard presented an explicit condition which he then used to prove that every residually finite congruence meet semidistributive variety is finitely based (see \cite{WillardTerms}).

The above description of the work of Kearnes and Szendrei and Lipparini cited above obscures what is perhaps the more important property that they discovered, which is that congruence meet semidistributive is in a sense perpendicular to \emph{abelianness}. Specifically, they showed that a variety $\Var$ is congruence meet semidistributive if and only if all commutator calculations in $\Var$ are \emph{neutral} (i.e. $[\alpha, \beta] = \alpha \wedge \beta$ for all congruences $\alpha, \beta$ of algebras $\mathbb{A} \in \Var$), which is the same thing as saying that there are no nontrivial abelian congruences of algebras in $\Var$. This has deep connections to finite domain fixed template constraint satisfaction problems. Barto and Kozik showed that local consistency methods solve $\Csp(\mathbf{B})$ for a fixed finite template of constraint relations $\mathbf{B}$ if and only if the algebra whose basic operations are the polymorphisms of $\mathbf{B}$ generates a congruence meet semidistributive variety (see \cite{BoundedWidth}). 

The connection between finite domain fixed template constraint satisfaction problems and Universal algebra has been remarkably fertile for both areas. On the one hand, the techniques of Universal algebra have provided a systematic framework in which to study the CSP, while on the other hand, the study of CSP within the Universal algebraic framework has produced new insights into the structure of locally finite varieties of algebras, which in some cases can even be extended to the class of all varieties. A striking example of this is that the insight of Siggers, that the class of locally finite Taylor varieties is characterized by a single finite package of identities (what is called a \emph{strong} Maltsev condition), was later extended by Ol\v{s}\'{a}k to the class of all Taylor varieties (see \cite{Siggers} and \cite{olsak-idempotent}, respectively). This was very exciting news at the time and it initiated a more intense scrutiny of the question of which among the classical Maltsev classes of varieties are strong Maltsev classes. Kozik, Krokhin,  Valeriote, and  Willard settle this question except for the class congruence meet semidistributive varieties (see \cite{KozikKrokBulWillardCharMaltsev}). We quote Ol\v{s}\'{a}k: `Meet semidistributive varieties are in a sense the last of the most important classes in universal algebra for which it is unknown whether it can be characterized by a strong Maltsev condition.' (see \cite{OlsakSDmeetterms}). 

We present here a proof that there is no finite package of identities which characterizes all congruence meet semidistributive varieties. There are two basic ingredients to the proof. In Section~\ref{sec:condition}, we define a specific sequence 
\[
\Sigma_1, \dots, \Sigma_n, \dots
\]
of finitely presented equational conditions which comprise a Maltsev condition for congruence meet semidistributivity. Then in Section~\ref{sec:notstrong}, we define a sequence of congruence meet semidistributive varieties
\[
\mathcal{W}_1, \dots, \mathcal{W}_l, \dots, 
\]
and argue syntactically that there does not exist $N$ so that each of the $\mathcal{W}_l$ has $\Sigma_N$-terms. We assume that the reader is familiar with the basics of Universal Algebra (standard references are \cite{McKenzieMcNultyTaylor} and \cite{BS}), the commutator (standard references are \cite{FreeseMcKenzie} and \cite{KearnesKiss}), and the notion of one variety $\mathcal{W}$ interpreting another variety $\mathcal{V}$ (see Chapter 6 of~\cite{ALV2} for a particularly nice exposition).

\section{A commutator related property characterizing congruence meet semidistributivity}\label{sec:condition}

Very broadly speaking, the commutator is a binary operation on the congruence lattice of an algebra which identifies the least congruence by which a quotient can be taken to obtain an algebra which has module structure (this should be treated as a slogan and not a definition). As noted in the introduction, it is well known that a variety of algebras $\mathcal{V}$ is congruence meet semidistributive if and only if $[\alpha, \beta] = \alpha \wedge \beta$ for all congruences $\alpha, \beta \in \Con(\mathbb{A})$ of algebras $\mathbb{A} \in \Var$ (see \cite{KearnesSzendreiRelTwoCommutators}). The commutator we refer to here is the \emph{term condition} commutator. In \cite{MoorheadTaylorSupNil}, it is shown that the neutrality of the term condition commutator is equivalent to the collapse of certain intervals in what we call the $(2)$-dimensional congruence lattice of any algebra belonging to $\Var$. This is the characterization of congruence meet semidistributivity that we will use to produce a Maltsev condition. We present a much abridged exposition of the theory of higher dimensional congruences and refer the reader to \cite{MoorheadTaylorSupNil}, \cite{MoorheadHigherKissTerms}, and \cite{MoorheadMaltsevComplexes} for a detailed development of the general picture, as well as \cite{keithagirossDifferenceTermChar} for a nontrivial application of the higher dimensional congruence perspective to different term varieties. 

The term condition can be thought of as the most basic among a collection of implications that must be satisfied in order for an algebra to possess module structure (or more generally a quasi-affine structure, as can be the case for Taylor algebras, see \cite{KearnesSzendreiRelTwoCommutators}). Formally, let $\mathbb{A}$ be an algebra and let $\theta_1$ and $\theta_2$ be congruences of $\mathbb{A}$. We define the algebra of $(\theta_1, \theta_2)$-matrices as 
\[
M(\theta_1, \theta_2) = 
\Sg_{A^{2^2}}
\left(
\left\{
\Square[x][x][y][y]: (x,y) \in \theta_1 
\right\}
\cup 
\left\{
\Square[x][y][x][y]: (x,y) \in \theta_2 
\right\}
\right).
\]

We then say that \emph{$\theta_1$ term condition centralizes $\theta_2$} if no matrix of $M(\theta_1, \theta_2)$ has one column which determines a pair of equal elements, while the opposite column determines a pair of unequal elements. The \emph{term condition} commutator is the least congruence $\delta$ that one can factor $\mathbb{A}$ by so that $\theta_1/ \delta$ centralizes $\theta_2/ \delta$. We denote this $\delta$ by $[\theta_1, \theta_2]_{TC}$.

The popularity of the term condition is likely due to the fact that it is computationally convenient to work with. In a congruence modular variety, it also possesses strong properties (see \cite{FreeseMcKenzie}), which mostly follow from the fact that the horizontal transitive closure of $M(\theta, \theta)$ (considered as a binary relation on its columns) is already vertically transitively closed (considered as a binary relation on its rows). It turns out that it is a good idea to enforce this $(2)$-dimensional transitivity at the outset. 

\begin{definition}
    Let $A$ be a set and let 
    $ R \subseteq A^{2^2}
    $. We say that $R $ is 
    \begin{enumerate}
        \item $(2)$-reflexive if $\Square[a][b][c][d] \in R$ implies $\Square[a][a][c][c], \Square[b][b][d][d], \Square[c][d][c][d] \Square[a][b][a][b] \in R$, 
        \item $(2)$-symmetric if $\Square[a][b][c][d] \in R$ implies $\Square[b][a][d][c], \Square[c][b][a][d]\in R$, 
        \item $(2)$-transitive if 
        \begin{itemize}
            \item  $\Square[a][b][c][d], \Square[c][d][e][f] \in R$ implies $\Square[a][b][e][f] \in R$
            \item $\Square[a][b][c][d], \Square[b][e][d][f] \in R$ implies $\Square[a][e][c][f] \in R$
        \end{itemize}
    \end{enumerate}
    We say that $R$ is a $(2)$-equivalence relation on $A$ if it is $(2)$-reflexive, $(2)$-symmetric, and $(2)$-transitive.  
    If $A$ is the universe of an algebra $\mathbb{A}$, we say that $R$ is 
    \begin{enumerate}
        \item A $(2)$-tolerance of $\mathbb{A}$ if it is $\mathbb{A}$-invariant, $(2)$-reflexive, and $(2)$-symmetric.
        \item A $(2)$-congruence of $\mathbb{A}$ if it is $\mathbb{A}$-invariant, $(2)$-refleive, $(2)$-symmetric, and $(2)$-transitive. 
    \end{enumerate}
\end{definition}

Given a set $X \subseteq A^{2^2}$, we denote by $\Eq_2(X)$ the least $(2)$-equivalence relation on $A$ that contains $X$. We may now refer to the $(2)$-congruence \emph{generated} by $X$, which is of course the least compatible $(2)$-equivalence of $\mathbb{A}$ containing the set $X$. Let us denote this $(2)$-congruence by $\Cg_2(X)$. We now define the relation 
\[
\Delta(\theta, \theta) = \Cg_2 \left(
\left\{
\Square[x][x][y][y]: (x,y) \in \theta_1 
\right\}
\cup 
\left\{
\Square[x][y][x][y]: (x,y) \in \theta_2 
\right\}
\right),
\]
for congruences $\theta_1$ and $\theta_2$ of an algebra $\mathbb{A}$. We say $\theta_1$ \emph{hypercentralizes} $\theta_2$ if no matrix of $\Delta(\theta_1, \theta_2)$ has one column which determines a pair of equal elements, while the opposite column determines a pair of unequal elements. We denote by $[\theta_1, \theta_2]_H$ the corresponding commutator. It is not hard to see that $\Delta(\theta, \theta)$ is the iterated transitive closure of $M(\theta, \theta)$, alternating between horizontal and vertical relational compositions, so it follows that $[\theta_1,\theta_2]_{TC} \leq [\theta_1, \theta_2]_H$. 

In fact, these two commutators are equal in a Taylor variety when evaluated at a constant pair of congruences (Theorem 4.9 of \cite{MoorheadTaylorSupNil}). Using this observation, it is possible to obtain the following characterization of congruence meet semidistributive varieties. For $\theta_1, \theta_2$ congruences of an algbera $\mathbb{A}$, we denote by $R(\theta_1, \theta_2) \leq A^{2^2}$ the set of all \emph{$(\theta_1, \theta_2)$-rectangles}, which are all matrices whose rows determine $\theta_1$-related pairs and columns determine $\theta_2$-related pairs. 

\begin{theorem}[Theorem 5.2 of \cite{MoorheadTaylorSupNil} ]\label{thm:mysdmeetcharacterization}
Let $\Var$ be a variety of algebras. The following are equivalent.
\begin{enumerate}
    \item $\Var$ is congruence meet semidistributive, and 
    \item $\Delta(\alpha, \alpha) = R(\alpha, \alpha)$ for any congruence $\alpha$ of an algebra $\mathbb{A} \in \Var$.
\end{enumerate}
\end{theorem}

We can sharpen the second condition of Theorem \ref{thm:mysdmeetcharacterization} by restricting the congruences that we consider to just a single principle congruence in the $2$-generated free algebra. We first introduce notation for the horizontal and vertical relational product of a set of matrices. Let $S \leq A^{2^2}$. We write 
\begin{itemize}
    \item $H(S) := \left\{ \Square[a][b][c][d]: \exists e,f \left( \Square[a][b][e][f], \Square[e][f][c][d] \in S \right)\right\}$ and
    \item $V(S) := \left\{ \Square[a][b][c][d]: \exists e,f \left(\Square[e][b][f][d] , \Square[a][e][c][f] \in S \right)\right\}$
\end{itemize}
to denote the \emph{horizontal and vertical} relational product of $S$ with itself. Compositions of these operations can be denoted with the usual exponent notation and the reader can easily check that the following hold when $S$ is a $(2)$-reflexive relation on $A$.
\begin{equation}
S \subseteq H(S)  \label{eq:scontainedinH(s)}
\end{equation}
\begin{equation}
S \subseteq V(S) \label{eq:scontainedinV(s)}
\end{equation}
Using (\ref{eq:scontainedinH(s)}) and (\ref{eq:scontainedinV(s)}), it is also straightforward to check that the following holds for any $(2)$-reflexive and $(2)$-symmetric relation $S$. 
\begin{equation}
\Eq_2(S) = \bigcup_{n \geq 0} (V\circ H)^n(S) \label{eq:delta=eltc}
\end{equation}

Now let $\Var$ be a variety and let $\mathbb{F}_{\Var}(x,y)$ be the free algebra for $\Var$ generated by $x$ and $y$. We define

\[
    E_{\Var}(x,y) :=\Sg_{\mathbb{F}_{\Var}(x,y)^{2^2}}
    \left(
    \left\{
 \Square[x][x][x][x], \Square[y][y][y][y], \Square[x][y][x][y], \Square[y][x][y][x], \Square[x][x][y][y], \Square[y][y][x][x]
    \right\}
    \right).
\]
We will refer to the elements of $E_{\Var}(x,y)$ as the \emph{$(x,y)$-elementary matrices} for $\Var$.
Now we can state the following theorem.

\begin{theorem}\label{thm:newsdmeetchar}
    Let $\Var$ be a variety of algebras. The following are equivalent.
    \begin{enumerate}
        \item $\Var$ is congruence meet semidistributive,
         \item $\Square[x][x][x][y] \in \Delta(\gamma, \gamma)$, where $\gamma$ is the congruence of the two generated free algebra $\mathbb{F}_{\Var}(x,y)$ in $\Var$ generated by the pair $(x,y)$, and
        \item $\Square[x][x][x][y] \in \Cg_2(E_{\Var}(x,y)) = \bigcup_{n \geq 0} (V\circ H)^n(E_{\Var}(x,y))$
    \end{enumerate}
\end{theorem}

\begin{proof}
    We first prove \emph{2.} holds if and only if \emph{3.} holds. This follows from the fact that $\Delta(\gamma, \gamma)$ is equal to $\Cg_2(E_{\Var}(x,y))$. Since the set of generators of $E_{\Var}(x,y)$ is a subset of the set of generators of $\Delta(\gamma, \gamma)$, clearly $\Cg_2(E_{\Var}(x,y)) \subseteq \Delta(\gamma, \gamma)$. On the other hand, the other generators of $\Delta(\gamma, \gamma)$ are easily seen to belong to the iterated transitive closure of 
   $E_{\Var}(x,y) $
    so $\Delta(\gamma, \gamma) \subseteq \Cg_2(E_{\Var}(x,y))$ as well. Finally, the equality 
    \[
    \Cg_2(E_{\Var}(x,y)) = \bigcup_{n \geq 0} (V\circ H)^n(E_{\Var}(x,y))
    \]
    follows from Equation (\ref{eq:delta=eltc}), since $E_{\Var}(x,y)$ is a $(2)$-tolerance of $\mathbb{F}_{\Var}(x,y)$.
    
    Now, \emph{1.} implies \emph{2.} follows from Theorem~\ref{thm:mysdmeetcharacterization}. So, suppose that \emph{2.} holds and let $\alpha$ be a congruence of an algebra $\mathbb{A} \in \Var$. Suppose that 
    \[
    \Square[a][b][c][d] \in R(\alpha, \alpha).
    \]
    Let $\Theta(a,b)$ and $\Theta(c,d)$ be the $\mathbb{A}$-invariant relations which are the images of $\Cg_2(E_{\Var}(x,y))$ under the respective mappings of $\mathbb{F}_{\Var}(x,y)$ to $\mathbb{A}$ which send $x$ to $a$ and $y$ to $b$ or $x$ to $c$ and $y$ to $d$. Clearly, both $\Theta(a,b) \subseteq \Delta(\alpha, \alpha)$ and $\Theta(c,d) \subseteq \Delta(\alpha, \alpha)$. Since we assume that \emph{2.} holds, we obtain that 
    \[
    \Square[a][a][a][b], \Square[c][c][c][d] \in \Delta(\alpha, \alpha).
    \]
    We also have the following matrix which is a generator of $\Delta(\alpha, \alpha)$:
    \[
    \Square[a][a][c][c] \in \Delta(\alpha, \alpha).
    \]
    Hence, we have the following three matrices (the first matrix is obtained by applying $(2)$-reflexivity):
   \[
    \Square[a][b][a][a], \Square[a][a][c][c] \Square[c][c][c][d] \in \Delta(\alpha, \alpha).
    \]
    Applying $(2)$-transitivity leads to the desired conclusion, i.e.\ that $\Delta(\alpha, \alpha)= R(\alpha, \alpha)$.
\end{proof}

Item \emph{3.} of Theorem~\ref{thm:newsdmeetchar} leads to a Maltsev condition that characterizes congruence meet semidistributivity. 
Since each element of $E(x,y)$ corresponds to some $6$-ary term of $\Var$ applied to the six generator matrices of $E(x,y)$, we can label each square witnessing condition \emph{3.} in Theorem~\ref{thm:newsdmeetchar} with a basic $6$-ary operation symbol and assert the obvious identities which guarantee a presentation of congruence meet semidistributivity. We denote by $\Sigma_n$ this package of identities. The reader can consult Figure \ref{fig:sigma2} for a diagram of the condition $\Sigma_2$.

\begin{figure}
    \centering
    \includegraphics[width=1\linewidth]{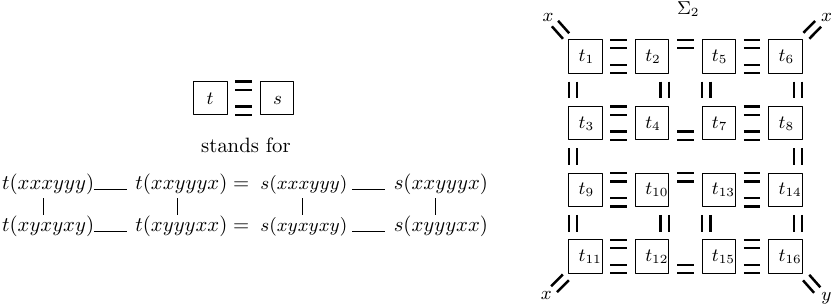}
    \caption{Diagram of the condition $\Sigma_2$}
    \label{fig:sigma2}
\end{figure}

\begin{theorem}
    A variety $\Var$ is congruence meet semidistributive if and only if there exists $n \geq 0$ and $6$-ary $\Var$-terms $t_1, \dots, t_{4^n}$ so that $\Var$ satisfies the identities of $\Sigma_n$ asserted for the $t_1, \dots, t_{4^n}$.
\end{theorem}

\begin{proof}
    If there exists an $n$ and such terms $t_1, \dots, t_{4^n}$ in $\Var$, then 
    \[
    \Square[x][x][x][y] \in \Delta(\gamma, \gamma),
    \]
    where $\gamma$ is the principle congruence generated by $(x,y)$ of the free algebra $\mathbb{F}_{\Var}(x,y)$, so $\Var $ is congruennce meet semidistributive by \emph{2.} implies \emph{1.} of Theorem~\ref{thm:newsdmeetchar}. On the other hand, if $\Var$ is congruence meet semidistributive, then by \emph{3.} of Theorem~\ref{thm:newsdmeetchar} there exists $n$ so that 
    \[
    \Square[x][x][x][y] \in (V\circ H)^n(E_{\Var}(x,y)) \subseteq \Delta(\gamma, \gamma).
    \]
    The above condition will produce a diagram with $4^n$-many matrices like the one shown in Figure~\ref{fig:sigma2}.
    Labeling each matrix of $E(x,y)$ by the $6$-ary $\Var$-term which generates it produces the terms which witness $\Sigma_n$. 
\end{proof}

Notice that Equations (\ref{eq:scontainedinH(s)}) and (\ref{eq:scontainedinV(s)}) guarantee that 
\[
(V\circ H)^n(E_{\Var}(x,y)) \subseteq (V\circ H)^{n+1}(E_{\Var}(x,y)),
\]
so any variety that has terms satisfying $\Sigma_n$ also has terms satisfying $\Sigma_{n+1}$. Hence, the conditions 
\[
\Sigma_1, \dots, \Sigma_n, \dots
\]
determine a descending chain in the interpretability lattice for varieties and therefore comprise a Maltsev condition for congruence meet semidistributivity. 
The remainder of this paper is dedicated to showing that this chain does not collapse, that is, there does not exist $N$ so that a variety is congruence meet semidistributive if and only if it satisfies $\Sigma_N$.

\section{Congruence meet semidistributivity is not strong}\label{sec:notstrong}

Let $l \geq 1$ and let $\tau_l = \{s_0, s_1, , \dots, s_{2l}, s_{2l+1} \}$ be a set of $4$-ary operation symbols. We denote by $\Lambda_l$ the following package of $\tau_l$ identities.

\begin{enumerate}
    \item $s_0(xxxx) = s_1(xxxx) = \dots =s_{2l}(xxxx) = s_{2l+1}(xxxx)=x$ 
    \item $s_0(yxxx) = x$ 
    \item $s_0(xyxx)= s_1(xyxx)$ and $s_0(yyxx)= s_1(yyxx)$
    \item $s_{2i+1}(xxyx)= s_{2(i+1)}(xxyx)$ and $s_{2i+1}(yxyx)= s_{2(i+1)}(yxyx)$ for all $0 \leq i < l$
    \item $s_{2i}(xyxx)= s_{2i+1}(xyxx)$ and $s_{2i}(yyxx)= s_{2i+1}(yyxx)$ for all $1 \leq i < l$
    \item $s_{2l}(xyxx)= s_{2+1}(xyxx)$
    \item $s_{2l}(yyxx)= s_{2l+1}(yyxx)$
    \item $s_{2l+1}(xyyy)= x$ 
\end{enumerate}

Let $\mathcal{W}_l$ be the variety of algebras which satisfy the identities presented by $\Lambda_l$. These identities are designed so that 
\[
\Square[x][x][x][y] \in \Delta(\gamma, \gamma),
\]
where, as usual, $\gamma$ is the congruence of $\mathbb{F}_{\mathcal{W}_l}(x,y)$ generated by the pair $(x,y)$. The reader can consult Figure~\ref{fig:horconditions} for an illustration of the $(x,y)$-elementary matrices which witness this fact. Hence by Theorem~\ref{thm:newsdmeetchar}, each of the $\mathcal{W}_l$ is a congruence meet semidistributive variety. 

\begin{figure}
    \centering
    \includegraphics[width=.9\linewidth]{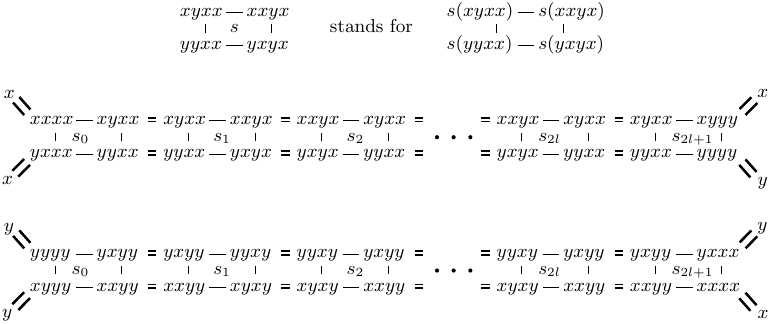}
    \caption{The condition $\Lambda_l$}
    \label{fig:horconditions}
\end{figure}

For each $0 \leq i \leq 2l+1$, we set $\tau_{l,i}  = \tau_l \setminus \{s_i\}$. We denote by $\Lambda_{l,i}$ the package of identities obtained from $\Lambda_l$ which only involve operation symbols belonging to $\tau_{l,i}$. We denote by $\mathcal{W}_{l,i}$ the variety of algebras that model the identities from $\Lambda_{l,i}$. 

\begin{lemma}\label{lem:projectionsinterpretation}
    Let $l \geq 1$, $0 \leq i \leq 2l+1$, and let $A$ be a nonempty set. The package of identities $\Lambda_{l,i}$ is modeled by an algebra with universe $A$ and whose basic operations are each a projection.
\end{lemma}
\begin{proof}

    We refer again to Figure~\ref{fig:horconditions}. If $i=0$, then all other operations can be interpreted as the first projection on $A$. If $i=2l+1$, then all other operations can be interpreted as the third projection on $A$. If $1 \leq i \leq 2l$, then the operation symbols $t_j$ for $0 \leq j <i$ can be interpreted as the third projection on $A$ and the others can be interpreted as the first projection on $A$. 
\end{proof}

Let us informally describe the idea of the proof that the chain $\Sigma_1, \dots, \Sigma_n, \dots$ does not terminate in some $\Sigma_N$. We pick a large $l$ with respect to $4^{N}$ (see Theorem~\ref{thm:notstrong}). If every congruence meet semidistributive variety has $\Sigma_N$-terms, then $\mathcal{W}_l$ will have $\Sigma_N$-terms and these terms will be witnessed by a collection of $4^N$-many $(x,y)$-elementary matrices belonging to $E_{\mathcal{W}_l}(x,y)$. Since we picked a large $l$, not all basic operation symbols can be used for the roots of the term trees which generate each of the $4^N$-many $(x,y)$-elementary squares. Applying Lemma~\ref{lem:projectionsinterpretation} will allow us to reduce the term complexity for each matrix by interpreting these outer basic operation symbols as projections. The critical step of the argument is to pass from the deductions we can make in $\mathcal{W}_l$ about term equality to the deductions we can make in some $\mathcal{W}_{l,i}$ (see Lemma~\ref{lemma:termsarefreegenerators}). The result then follows from an induction on the complexity of terms which generate $E_{\mathcal{W}_l}(x,y)$. Of course, this kind of argument will not work in general and the fact that it works here relies on the good behavior of the $\mathcal{W}_l$-terms.

Our task now is to build the $2$-generated free algebra for $\mathcal{W}_l$. Actually, we will recursively build a set of $\tau_l$-terms which all together will comprise a set of distinct representatives for the equivalence classes of the fully invariant congruence by which the term algebra is factored to produce the free algebra. At each stage we will also define the action of the $\tau_l$ basic operations which satisfy all possible instances of $\Lambda_l$-identities, thereby producing a $\tau_l$-algebra which belongs to $\mathcal{W}_l$. 

\begin{figure}[!htb]
\centering
    \begin{tabular}{c|cccccc}
     & $s_0^{\mathbb{F}_l^1}$&$s_1^{\mathbb{F}_l^1}$ & $s_2^{\mathbb{F}_l^1}$ & \dots & $s_{2l}^{\mathbb{F}_l^1}$ & $s_{2l+1}^{\mathbb{F}_l^1}$ \\
         \hline
        $xxxx$ &$x$ & $x$ &$x$ & \dots &$x$ & $x$\\
        $xxxy$ &$s_0(xxxy)$ & $s_1(xxxy)$& $s_2(xxxy)$ & \dots & $s_{2l}(xxxy)$& $s_{2l+1}(xxxy)$ \\
        $xxyx$ &$s_0(xxyx)$ & $s_1(xxyx)$& $\leftarrow \mathbf{s_1(xxyx)}$ & \dots & $s_{2l}(xxyx)$& $s_{2l+1}(xxyx)$ \\
        $xxyy$ &$s_0(xxyy)$ & $\leftarrow \mathbf{s_0(xxyy)}$& $s_2(xxyy)$ & \dots & $s_{2l}(xxyy)$& $\leftarrow \mathbf{s_{2l}(xxyy)}$ \\
        $xyxx$ &$s_0(xyxx)$ & $\leftarrow \mathbf{s_0(xyxx)}$& $s_2(xyxx)$ & \dots & $s_{2l}(xyxx)$& $\leftarrow \mathbf{s_{2l}(xyxx)}$ \\
        $xyxy$ &$s_0(xyxy)$ & $s_1(xyxy)$& $\leftarrow \mathbf{s_1(xyxy)}$ & \dots & $s_{2l}(xyxy)$& $s_{2l+1}(xyxy)$ \\
        $xyyx$ &$s_0(xyyx)$ & $s_1(xyyx)$& $s_2(xyyx)$ & \dots & $s_{2l}(xyyx)$& $s_{2l+1}(xyyx)$ \\
        $xyyy$ &$y$ & $s_1(xyyy)$& $s_2(xyyy)$ & \dots & $s_{2l}(xyyy)$& $x$ \\
        $yxxx$ &$x$ & $s_1(yxxx)$& $s_2(yxxx)$ & \dots & $s_{2l}(yxxx)$& $y$ \\
        $yxxy$ &$s_0(yxxy)$ & $s_1(yxxy)$& $s_2(yxxy)$ & \dots & $s_{2l}(yxxy)$& $s_{2l+1}(yxxy)$ \\
        $yxyx$ &$s_0(yxyx)$ & $s_1(yxyx)$& $\leftarrow \mathbf{s_1(yxyx)}$ & \dots & $s_{2l}(yxyx)$& $s_{2l+1}(yxyx)$ \\
        $yxyy$ &$s_0(yxyy)$ & $\leftarrow \mathbf{s_0(yxyy)}$& $s_2(yxyy)$ & \dots & $s_{2l}(yxyy)$& $\leftarrow \mathbf{s_{2l}(yxyy)}$ \\
        $yyxx$ &$s_0(yyxx)$ & $\leftarrow \mathbf{s_0(yyxx)}$& $s_2(yyxx)$ & \dots & $s_{2l}(yyxx)$& $\leftarrow \mathbf{s_{2l}(yyxx)}$ \\
        $yyxy$ &$s_0(yyxy)$ & $s_1(yyxy)$& $\leftarrow \mathbf{s_1(yyxy)}$ & \dots & $s_{2l}(yyxy)$& $s_{2l+1}(yyxy)$ \\
        $yyyx$ &$s_0(yyyx)$ & $s_1(yyyx)$& $s_2(yyyx)$ & \dots & $s_{2l}(yyyx)$& $s_{2l+1}(yyyx)$ \\
        $yyyy$ & $y$&$y$ &$y$ &\dots &$y$ &$y$ \\
    \end{tabular}
    \caption{The operation definitions for the partial $\tau_l$-algebra $\mathbb{F}_l^1$}
    \label{fig:heightonetermx}

\end{figure}

\begin{figure}[!htb]
\centering
\begin{tabular}{c}
Input tuples $(a,b,c,d) \in F^{k}_l$ satisfy $\{a,b,c,d \} \cap (F^{k}_l \setminus F^{k-1}_l) \neq \emptyset$ 
\end{tabular}

    \begin{tabular}{c|cccccc}
    \hline
     & $s_0^{\mathbb{F}_l^{k+1}}$&$s_1^{\mathbb{F}_l^{k+1}}$ & $s_2^{\mathbb{F}_l^{k+1}}$ & \dots & $s_{2l}^{\mathbb{F}_l^{k+1}}$ & $s_{2l+1}^{\mathbb{F}_l^{k+1}}$ \\
         \hline
        $pppp$ &$p$ & $p$ &$p$ & \dots &$p$ & $p$\\
        $pppq$ &$s_0(pppq)$ & $s_1(pppq)$& $s_2(pppq)$ & \dots & $s_{2l}(pppq)$& $s_{2l+1}(pppq)$ \\
        $ppqp$ &$s_0(ppqp)$ & $s_1(ppqp)$& $\leftarrow \mathbf{s_1(ppqp)}$ & \dots & $s_{2l}(ppqp)$& $s_{2l+1}(ppqp)$ \\
        $ppqq$ &$s_0(ppqq)$ & $\leftarrow \mathbf{s_0(ppqq)}$& $s_2(ppqq)$ & \dots & $s_{2l}(ppqq)$& $\leftarrow \mathbf{s_{2l}(ppqq)}$ \\
        $pqpp$ &$s_0(pqpp)$ & $\leftarrow \mathbf{s_0(pqpp)}$& $s_2(pqpp)$ & \dots & $s_{2l}(pqpp)$& $\leftarrow \mathbf{s_{2l}(pqpp)}$ \\
        $pqpq$ &$s_0(pqpq)$ & $s_1(pqpq)$& $\leftarrow \mathbf{s_1(pqpq)}$ & \dots & $s_{2l}(pqpq)$& $s_{2l+1}(pqpq)$ \\
        $pqqp$ &$s_0(pqqp)$ & $s_1(pqqp)$& $s_2(pqqp)$ & \dots & $s_{2l}(pqqp)$& $s_{2l+1}(pqqp)$ \\
        $pqqq$ &$q$ & $s_1(pqqq)$& $s_2(pqqq)$ & \dots & $s_{2l}(pqqq)$& $p$ \\
    \end{tabular}
    
\vspace{1cm}
\begin{tabular}{c}
Input tuples $(a,b,c,d) \in F^{k}_l$ satisfy $\{a,b,c,d \} \cap (F^{k}_l \setminus F^{k-1}_l) \neq \emptyset $ and 
$|\{a,b,c,d\} |\geq 3$
\end{tabular}

    \begin{tabular}{c|cccccc}
    \hline
     & $s_0^{\mathbb{F}_l^{k+1}}$&$s_1^{\mathbb{F}_l^{k+1}}$ & $s_2^{\mathbb{F}_l^{k+1}}$ & \dots & $s_{2l}^{\mathbb{F}_l^{k+1}}$ & $s_{2l+1}^{\mathbb{F}_l^{k+1}}$ \\
         \hline
        $abcd$ &$s_0(abcd)$ & $s_1(abcd)$& $s_2(abcd)$ & \dots & $s_{2l}(abcd)$& $s_{2l+1}(abcd)$ \\
    \end{tabular}
\caption{The operation definitions for the partial $\tau_l$-algebra $\mathbb{F}_l^{k+1}$}
    \label{fig:recursivealgebradef2}
\end{figure}

For the basis of the recursion, we set $F_l^0 = \{x,y \}$ and $F_l^1$ all $\tau_l$-terms which appear as entries in the table provided in Figure~\ref{fig:heightonetermx}. We define a partial $\tau_l$-algebra $\mathbb{F}_l^1$, where each basic operation symbol $r \in \tau_l$ interprets as a partial operation 
\[
r^{\mathbb{F}_l^1}: (F_1^0)^4 \to F_l^1
\]
as specified by the table. Note that we indicate with bold and an arrow when one of the `connecting' identities is used (identities 3.-7. in the definition of $\Lambda_l$ given at the beginning of the section).

Now we proceed recursively and specify a sequence of sets $F^0_l \subseteq F^1_l \subseteq  \dots \subseteq F^{k-1}_l \subseteq  F^k_l \dots$, where each $F^k_l$ is the domain of a partial $\tau_l$-algebra $\mathbb{F}_l^k$ with all $\tau_l$ operations defined on $(F^{k-1}_l)^4$, for every $k \geq 1$. Given the partial $\tau_l$-algebra $\mathbb{F}^k_l$, the partial $\tau_l$-algebra $\mathbb{F}^{k+1}_l$ is defined by extending the operation $r^{\mathbb{F}^{k}_l}$ to $(F^k_l)^l$ and defining $F^{k+1}$ to be union of $F^{k}_l$ with a set comprised of some appropriately chosen set of terms which guarantee that all $\Lambda_l$-identities are satisfied. Indeed, this is accomplished with the data provided in the tables appearing in Figure~\ref{fig:recursivealgebradef2}. The set $F^{k+1}_l$ is taken to be the union of $F^k_l$ with the set of all terms which appear in one of the operation tables given in the figure and each basic $\tau_l$-operation is defined by extending its definition in $\mathbb{F}^k_l$ as specified in the tables given in the figure.  

We now define $F_l = \bigcup_{1 \leq k} F_l^k$ to be the domain of the $\tau_l$-algebra $\mathbb{F}_l$, where the image of a tuple $(a,b,c,d)$ under a basic $\tau_l$-operation is obtained by consulting its image in $\mathbb{F}^k_l$, with $k$ minimal so that $\{a,b,c,d\} \in F^k_l$. 

\begin{lemma}\label{lem:freealgebra}
    Let $l \geq 1$. The algebra $\mathbb{F}_l$ belongs to $\mathcal{W}_l$ and is freely generated by the set $\{x,y\}$.  
\end{lemma}

\begin{proof}
    The theorem will follow from the following claim.

    \begin{claim}
        Let $k \geq 1$. The following hold.
        \begin{enumerate}
            \item The partial $(\tau_l)$-algebra $\mathbb{F}_l^k$ satisfies all $\Lambda_l$-identities over $F^{k-1}_l \subseteq F^k_l$.
            \item Each $\phi: \{x,y\} \to A$, for $\mathbb{A} \in \mathcal{W}_l$ extends to a partial $\tau_l$-algebra homomorphism $\phi_k: \mathbb{F}^k_l \to \mathbb{A}$. 
            
        \end{enumerate}

    \end{claim}
    \begin{claimproof}
        The proof of the claim proceeds inductively on $k \geq1$. We first establish the basis of the induction, where $k=1$. The reader can consult Figure~\ref{fig:heightonetermx} to assure themselves that $\mathbb{F}^1_l$ satisfies all $\Lambda_l$-identities over $F^0_l = \{x,y\}$, so \emph{1.} holds. To establish \emph{2.}, take an algebra $\mathbb{A} \in \mathcal{W}_l$ and a mapping $\phi: \{x,y\} \to A$. We define for a term $p(x,y) \in F^1_l$ the mapping.
        \[
        \phi(p(x,y))_1 = 
            p^{\mathbb{A}}(\phi(x), \phi(y)).
        \]
     Since $\mathbb{A}$ is assumed to model the identities in $\Lambda_l$, it follows that $\phi_k$ is a homomorphism of partial $\tau_l$-algebras that extends $\phi$, since only the $\Lambda_l$-identities are used to define the basic $\tau_l$-operations in the table in Figure~\ref{fig:heightonetermx}.

     Now suppose the claim holds for $k \geq 1$. The inductive argument is essentially identical to the basis. It follows from the definition of $\mathbb{F}^{k+1}$ and the inductive assumption that all $\Lambda_l$ identities are satisfied in $\mathbb{F}^{k+1}_l$ for tuples with entries ranging over $F^{k-1}_l \subseteq F^{k+1}_l$. Additionally, it follows from the tables given in Figure~\ref{fig:recursivealgebradef2} that the $\Lambda_l$-identities are also satisfied for all other tuples with entries ranging over $F^{k}_l$. So, \emph{1.} of the claim is established. 

     To establish \emph{2.}, take an algebra $\mathbb{A} \in \mathcal{W}_l$ and a mapping $\phi: \{x,y\} \to A$. We inductively suppose that there exists a $\phi_k: \mathbb{F}^k_l \to \mathbb{A}$ which extends $\phi$ and is a partial $\tau_l$-homomorphism. We define for a term $p(x,y) \in F^{k+1}_l$ the following mapping.
        \[
        \psi(p(x,y))_{k+1} = 
        \begin{cases}
            \phi_k(p(x,y)) &\text{if $p(x,y) \in F^{k}_l$, otherwise}\\
            r^{\mathbb{A}}(\phi_k(a), \phi_k(b), \phi_k(c), \phi_k(d)) &\text{for $r \in \tau_l$ such that $p(x,y) = r(a,b,c,d)$. }\\
        \end{cases}
        \]
        Since $\phi_k$ is assumed to satisfy the homomorphism property over its domain, we need only verify that $\phi_{k+1}$ satisfies the homomorphism property for terms $p(x,y) \in F^{k+1}_l \setminus F^k_l$, but this argument is identical to the one given for the basis, except that we of course now refer to Figure~\ref{fig:recursivealgebradef2}.
     \end{claimproof}
     
    The lemma now follows from the definition of $\mathbb{F}_l$. Indeed, $\mathbb{F}_l$ satisfies the identities of $\Lambda_l$, since any concrete evaluation of a basic $\tau_l$-operation in $\mathbb{F}_l$ occurs within some $F^k_l \subseteq F_l$. hence \emph{1.} of the claim applies. To see $\mathbb{F}_l$ is free, we take any $\phi: \{x,y \} \to \A$ for $\mathbb{A} \in \mathcal{W}_l$ and set $\overline{\phi} = \bigcup_{1 \leq k } \phi_k$. Clearly, $\overline{\phi}$ extends $\phi$ and is a $\tau_l$-algebra homomorphism. \qedhere
\end{proof}

Having established a representation of the $2$-generated free algebra for $\mathcal{W}_l$, we move to the main idea that makes the argument work.

\begin{lemma}\label{lemma:termsarefreegenerators}
    Let $l \geq 1$,  and $m \geq 1$. Consider a set of $\tau_l$-terms of the form
    \[
    T= \{r_j(a_j,b_j,c_j,d_j) : r_j \in \tau_{l} \text{ and } a_j, b_j, c_j, d_j \in F_l, \text{ for } 1\leq j \leq m\}
    \] 
    and let 
    \[Z = \{ z: \text{there exists $r_j(a_j,b_j,c_j,d_j) \in T$ with $ r_j = s_z$} \}
    \]
    be the set of indices of basic $\tau_l$-operation symbols which appear as the outer symbol for a term in $T$. 
    If there exists $0\leq i \leq 2l+1$ such that $|z -i | \geq 2$ for all $z \in Z$, then 
    \[
    r_{j_1}^{\mathbb{F}_l}(a_{j_1},b_{j_1},c_{j_1},d_{j_1}) = r_{j_2}^{\mathbb{F}_l}(a_{j_2},b_{j_2},c_{j_2},d_{j_2}) \iff 
    r_{j_1}^{\mathbb{G}}(a_{j_1},b_{j_1},c_{j_1},d_{j_1}) = r_{j_2}^{\mathbb{G}}(a_{j_2},b_{j_2},c_{j_2},d_{j_2}),
    \]
    where $\mathbb{G} = \mathbb{F}_{\mathcal{W}_{l,i}}(\bigcup_{1 \leq j \leq m} \{a_j, b_j, c_j, d_j \})$ is the algebra for $\mathcal{W}_{l,i}$ freely generated by 
    \[\bigcup_{1 \leq j \leq m} \{a_j, b_j, c_j, d_j \},\]
    for all $r_{j_1}^{\mathbb{F}_l}(a_{j_1},b_{j_1},c_{j_1},d_{j_1}), r_{j_2}^{\mathbb{F}_l}(a_{j_2},b_{j_2},c_{j_2},d_{j_2}) \in T$.
\end{lemma}

\begin{proof}
    We first informally describe what the theorem statement means, which is that an equality in $\mathbb{F}_l$ between the outputs of basic $\tau_{l,i}$-operations (recall these are all basic $\tau_l$-operations except for $s_i$) holds if and only if the equality can be deduced from the identities in $\Lambda_{l,i}$ applied exactly to the basic operation symbols in question. Essentially, the full strength of $\Lambda_l$ has already been applied to the subterms $a_j,b_j,c_j,d_j$ in the recursive definition of $\mathbb{F}_l$, so even if these subterms contain all basic $\tau_l$-symbols, they can be treated like free generators in the variety $\mathcal{W}_{l,i}$, provided the set $Z$ of indices of basic symbols is sufficiently sparse among the possible indices $\{0, \dots, 2l+1\}$.

    We fix $Z$ and $i$ at the outset. The proof proceeds inductively on the least $k$ so that the maximal subterms $\bigcup_{1 \leq j \leq m} \{a_j, b_j, c_j, d_j \} $ are a subset of $F^k_l$. The basis is the case $k=0$, so all of the $a_j, b_j, c_j, d_j \in \{x,y\}$. We again refer to Figure~\ref{fig:heightonetermx}. Suppose that $i=0$, so none of the $r_j$ are equal to $s_0$. Now, the only potential candidates for an equality among the terms which appear in the other columns are when the outputs belong to the set
    \[
    \{x,y\} \cup O \cup E,
    \]
    where 
    \begin{align*}
    O =  &\bigcup_{0 \leq k < l} \{ s_{2k+1}(xxyx),s_{2k+1}(xyxy), s_{2k+1}(yxyx), s_{2k+1}(yyxy)\} \text{ and }\\ E= &\bigcup_{1 \leq k \leq l} \{s_{2k}(xxyy),  s_{2k}(xyxx), s_{2k}(yxyy), s_{2k}(yyxx) \}.
    \end{align*}
    If we have an equality where
    \[
    r_{j_1}^{\mathbb{F}_l}(a_{j_1},b_{j_1},c_{j_1},d_{j_1}) = r_{j_2}^{\mathbb{F}_l}(a_{j_2},b_{j_2},c_{j_2},d_{j_2}) \in \{x,y\}
    \]
    for $a_{j_1},b_{j_1}, c_{j_1}, d_{j_1}, a_{j_2},b_{j_2}, c_{j_2}, d_{j_2} \in \{x,y\}$,
    then this equality can be deduced from the $\Lambda_l$-identities which involve only $r_{j_1}$ or $r_{j_2}$, hence it can be deduced from $\Lambda_{l,i}$. The equalities of the form
     \begin{align*}
    r_{j_1}^{\mathbb{F}_l}(a_{j_1},b_{j_1},c_{j_1},d_{j_1}) = r_{j_2}^{\mathbb{F}_l}(a_{j_2},b_{j_2},c_{j_2},d_{j_2}) \in  O \cup E
    \end{align*}
    can be deduced from a single $\Lambda_l$-identity which involves only $r_{j_1}$ and $r_{j_2}$, hence they are also deducible from $\Lambda_{l,i}$. The backwards implication is easy, since if \[r_{j_1}^{\mathbb{G}}(a_{j_1},b_{j_1},c_{j_1},d_{j_1}) = r_{j_2}^{\mathbb{G}}(a_{j_2},b_{j_2},c_{j_2},d_{j_2})\] holds, then it holds in the $\tau_{l,i}$-reduct of $\mathbb{F}_l$, since this reduct belongs to $\mathcal{W}_{l,i}$.

   Now we prove the inductive step. Assume that the theorem statement holds for the fixed $Z$ and $i$ when the subterms $\bigcup_{1 \leq j \leq m} \{a_j, b_j, c_j, d_j \} \subseteq F^k_l$ for some $k \geq 0$. Suppose we have a set of terms $T$ as in the theorem statement so that the subterms $\bigcup_{1 \leq j \leq m} \{a_j, b_j, c_j, d_j \} \subseteq F^{k+1}_l$. Let $X = F^{k}_l \cap \left(\bigcup_{1 \leq j \leq m} \{a_j, b_j, c_j, d_j \} \right) $. If we have an equality where
    \[
    r_{j_1}^{\mathbb{F}_l}(a_{j_1},b_{j_1},c_{j_1},d_{j_1}) = r_{j_2}^{\mathbb{F}_l}(a_{j_2},b_{j_2},c_{j_2},d_{j_2}) 
    \]
    where $ \{a_{j_1},b_{j_1},c_{j_1},d_{j_1}, a_{j_2},b_{j_2},c_{j_2},d_{j_2}\} \subseteq X$, then the inductive hypothesis implies that 
     \[
    r_{j_1}^{\mathbb{G}'}(a_{j_1},b_{j_1},c_{j_1},d_{j_1}) = r_{j_2}^{\mathbb{G}'}(a_{j_2},b_{j_2},c_{j_2},d_{j_2}),
    \]
    where $\mathbb{G}' = \mathbb{F}_{\mathcal{W}_{l,i}}(X)$. Since $\mathbb{G}'$ is a freely generated subalgebra of $\mathbb{G}= \mathbb{F}_{\mathcal{W}_{l,i}}(\bigcup_{1 \leq j \leq m} \{a_j, b_j, c_j, d_j \} )$, it follows that 
     \[
    r_{j_1}^{\mathbb{G}}(a_{j_1},b_{j_1},c_{j_1},d_{j_1}) = r_{j_2}^{\mathbb{G}}(a_{j_2},b_{j_2},c_{j_2},d_{j_2}). 
    \]

    Hence, we are left to consider the situation where we have an equality 
    \[
    r_{j_1}^{\mathbb{F}_l}(a_{j_1},b_{j_1},c_{j_1},d_{j_1}) = r_{j_2}^{\mathbb{F}_l}(a_{j_2},b_{j_2},c_{j_2},d_{j_2}) 
    \]
    and at least one element of  $\{a_{j_1},b_{j_1},c_{j_1},d_{j_1}, a_{j_2},b_{j_2},c_{j_2},d_{j_2}\}$ belongs to $F^{k+1}_l \setminus F^k_l$.  Without loss of generality, we suppose that this element belongs to $\{a_{j_1}, b_{j_1}, c_{j_1}, d_{j_1}\}$.
    We consider the following cases.
    \begin{itemize}
       
            \item Both $r_{j_1}^{\mathbb{F}_l}(a_{j_1},b_{j_1},c_{j_1},d_{j_1}) \in \{a_{j_1},b_{j_1},c_{j_1},d_{j_1}\}$ and $r_{j_2}^{\mathbb{F}_l}(a_{j_2},b_{j_2},c_{j_2},d_{j_2}) \in \{a_{j_2},b_{j_2},c_{j_2},d_{j_2}\}$. In this case the operations $r_{j_1}^{\mathbb{F}_l}$ and $r_{j_2}^{\mathbb{F}_l}$ each apply a single $\Lambda_l$-identity which only involves $r_{j_1}$ or $r_{j_2}$, respectively. These identities each belong to $\Lambda_{l,i}$, hence the equality 
             \[
    r_{j_1}^{\mathbb{G}}(a_{j_1},b_{j_1},c_{j_1},d_{j_1}) = r_{j_2}^{\mathbb{G}}(a_{j_2},b_{j_2},c_{j_2},d_{j_2}) 
    \]
    holds. 
    \item Suppose $r_{j_1}^{\mathbb{F}_l}(a_{j_1},b_{j_1},c_{j_1},d_{j_1}) \in \{a_{j_1},b_{j_1},c_{j_1},d_{j_1}\}$ and $r_{j_2}^{\mathbb{F}_l}(a_{j_2},b_{j_2},c_{j_2},d_{j_2}) \notin \{a_{j_2},b_{j_2},c_{j_2},d_{j_2}\}$. Let $z \in Z$ be such that $r_{j_2} = s_z$. Then 
    \[
    r_{j_2}^{\mathbb{F}_l}(a_{j_2},b_{j_2},c_{j_2},d_{j_2}) = s_z^{\mathbb{F}_l}(a_{j_2},b_{j_2},c_{j_2},d_{j_2}) \in \{s_z(a_{j_2},b_{j_2},c_{j_2},d_{j_2}), s_{z-1}(a_{j_2},b_{j_2},c_{j_2},d_{j_2})\}.
    \]
     Since we assume that $|z-i| \geq 2$ for all $z \in Z$, either of these cases follows from an identity in $\Lambda_{l,i}$, so we again have that 
      \[
    r_{j_1}^{\mathbb{G}}(a_{j_1},b_{j_1},c_{j_1},d_{j_1}) = r_{j_2}^{\mathbb{G}}(a_{j_2},b_{j_2},c_{j_2},d_{j_2}) 
    \]
    holds. 
    \item Suppose $r_{j_1}^{\mathbb{F}_l}(a_{j_1},b_{j_1},c_{j_1},d_{j_1}) \notin \{a_{j_1},b_{j_1},c_{j_1},d_{j_1}\}$. Then it must be that at least one element of $\{a_{j_2},b_{j_2},c_{j_2},d_{j_2}\}$ belongs to $F^{k+1}_l \setminus F^k_l$ and $r_{j_2}^{\mathbb{F}_l}(a_{j_2},b_{j_2},c_{j_2},d_{j_2}) \notin \{a_{j_2},b_{j_2},c_{j_2},d_{j_2}\}$ also, as otherwise $r_{j_1}^{\mathbb{F}_l}(a_{j_1},b_{j_1},c_{j_1},d_{j_1})  \in F^{k+2}_l \setminus F^{k+1}_l$ and $ r_{j_2}^{\mathbb{F}_l}(a_{j_2},b_{j_2},c_{j_2},d_{j_2}) \in F^{k+1}_l$. We consult Figure~\ref{fig:recursivealgebradef2} and see that, again, $\Lambda_{l,i}$ identities are adequate, hence
      \[
    r_{j_1}^{\mathbb{G}}(a_{j_1},b_{j_1},c_{j_1},d_{j_1}) = r_{j_2}^{\mathbb{G}}(a_{j_2},b_{j_2},c_{j_2},d_{j_2}) 
    \]
    holds.

    \end{itemize}
    The reverse implication of the theorem statement follows from the same argument that we gave for the basis of the induction. Therefore, the theorem is proved. 

\end{proof}

\begin{theorem}\label{thm:notstrong}
    Let $N\geq 0$ and let $l > 2\cdot4^N$. The variety $\mathcal{V}_N$ does not interpret in the variety $\mathcal{W}_l$.
\end{theorem}

\begin{proof}
    By definition, $\Var_N$ is the class of algebras in the signature $\{t_1, \dots, t_{4^N} \}$ of $6$-ary operation symbols which satsify the $\Sigma_N$-identities. Suppose that $\Var_N$ interprets in $\mathcal{W}_l$, i.e., there are $\mathcal{W}_l$ terms $\{t_1, \dots, t_{4^N}\}$ which satisfy the $\Sigma_N$-identities. By the definition of $\Sigma_N$ (see \emph{3.} of Theorem~\ref{thm:newsdmeetchar} and the discussion afterwords) and the fact that $\mathbb{F}_l \in \mathcal{W}_l$ (see Lemma~\ref{lem:freealgebra}), it follows that 
    \begin{equation}\label{eq:condition}
    \Square[x][x][x][y] \in (V \circ H)^N(E_{\mathcal{W}_l}(x,y)),
    \end{equation}
    where 
    \[
    E_{\mathcal{W}_l}(x,y) :=\Sg_{(\mathbb{F}_{l})^{2^2}}
    \left(
    \left\{
 \Square[x][x][x][x], \Square[y][y][y][y], \Square[x][y][x][y], \Square[y][x][y][x], \Square[x][x][y][y], \Square[y][y][x][x]
    \right\}
    \right).
\]

We will show that this is impossible by an induction on a recursively defined sequence $E_0 \subseteq E_1 \subseteq \dots \subseteq E_k \subseteq \dots $ whose union is $ E_{\mathcal{W}_l}(x,y)$. Indeed, we set
\begin{align*}
E_0 &=  \left\{
 \Square[x][x][x][x], \Square[y][y][y][y], \Square[x][y][x][y], \Square[y][x][y][x], \Square[x][x][y][y], \Square[y][y][x][x]
    \right\} \text{ and }\\
E_{k+1} &= \{ r^{(\mathbb{F}_{l})^{2^2}}( \alpha, \beta, \gamma, \delta) : r\in \tau_l \text{ and } \alpha, \beta, \delta, \gamma \in E_k \}  \text{ for $k \geq 0 $}. 
\end{align*}
Since all basic $\tau_l$-operations are idempotent in $\mathbb{F}_l$, it follows that $E_k \subseteq E_{k+1}$. Notice that $E_k \subseteq (F^k_l)^{2^2}$, that is, all entries of matrices in $E_k$ are elements of $F^k_l$. Obviously, we have
\[
 E_{\mathcal{W}_l}(x,y) = \bigcup_{k \geq 0 } E_k,
\]
hence if the set membership of~(\ref{eq:condition}) holds, then there exists a minimal $k$ so that 
\[
 \Square[x][x][x][y] \in (V \circ H)^N(E_{k}).
\]
We note that $k$ is obviously not equal to $0$. 
By the definition of the vertical and horizontal relational composition operators $V$ and $H$, there exist $4^N$-many matrices $\zeta_1, \dots, \zeta_{4^N}$ belonging to $E_k$ which can be arranged in a diagram with equalities between their coordinates like those in Figure~\ref{fig:sigma2}. By definition of $E_k$, each of these matrices is of the form \[
r_1^{(\mathbb{F}_{l})^{2^2}}(\alpha_1, \beta_1, \gamma_1, \delta_1), \dots, r_{4^N}^{(\mathbb{F}_{l})^{2^2}}(\alpha_{4^N}, \beta_{4^N}, \gamma_{4^N}, \delta_{4^N}),
\]
where each $r_{w} \in \tau_l$ and $\alpha_w, \beta_w, \gamma_w, \delta_w \in E_{k-1}$, i.e.\ each matrix is the output of a basic $\tau_l$-operation with each evaluation coming from $E_{k-1}$. Let $Z$ be the set of all indices of the basic $\tau_l$-operation symbols which are used for the $\zeta_1, \dots, \zeta_{4^N}$ and let 
\[
T= \{s_z(a_j,b_j,c_j,d_j) : z \in Z \text{ and } a_j, b_j, c_j, d_j \in F_k\}.
\]
 We assume that $l > 2\cdot4^N$, so all of the hypotheses of Lemma~\ref{lemma:termsarefreegenerators} apply. Let $0 \leq i \leq 2l+1$ be some index so that $|z-i| \geq 2$ for all $z \in Z$. 

It follows that, for all $s_{z_1}^{\mathbb{F}_l}(a_{j_1},b_{j_1},c_{j_1},d_{j_1}) $ and $ s_{z_2}^{\mathbb{F}_l}(a_{j_2},b_{j_2},c_{j_2},d_{j_2})$ in $T$,  \[
    s_{z_1}^{\mathbb{F}_l}(a_{j_1},b_{j_1},c_{j_1},d_{j_1}) = s_{z_2}^{\mathbb{F}_l}(a_{j_2},b_{j_2},c_{j_2},d_{j_2}) \iff 
    s_{z_1}^{\mathbb{G}}(a_{j_1},b_{j_1},c_{j_1},d_{j_1}) = s_{z_2}^{\mathbb{G}}(a_{j_2},b_{j_2},c_{j_2},d_{j_2}),
    \]
    where $\mathbb{G} = \mathbb{F}_{\mathcal{W}_{l,i}}(F^k_l)$. By Lemma~\ref{lem:projectionsinterpretation}, there exists a $\tau_{l,i}$-algebra $\mathbb{S}$ with universe $F^k_l$ in which each basic $\tau_{l,i}$-symbol interprets as a projection. If we apply this interpretation by projections to each of the $r_1, \dots, r_{4^N} \in \tau_{l,i}$ which produces one of the matrices $\zeta_1, \dots, \zeta_{4^N}$, it follows that
    \[
    \Square[x][x][x][y] \in (V \circ H)^N\left(\bigcup_{1 \leq w \leq 4^N} \{\alpha_w, \beta_w, \gamma_w, \delta_w \}\right),
    \]
    which is a contradiction to the minimality of $k$.  
    
\end{proof}

\begin{corollary}
    The class of congruence meet semidistributive varieties is not characterized by a strong Maltsev condition.
\end{corollary}
\begin{proof}
    Suppose towards a contradiction that there exists a finite package of identities $\Sigma_{SD}$ so that a variety $\Var$ is congruence meet semidistributive if and only if $\Var$ has $\Sigma_{SD}$-terms. By assumption, the variety $\Var_{\Sigma_{SD}}$ presented by $\Sigma_{SD}$ is  congruence meet semidistributive, hence it has $\Sigma_N$-terms for some positive $N$. On the other hand, by Theorem~\ref{thm:notstrong}, there exists $l$ so that $\mathcal{W}_l$ does not have $\Sigma_N$-terms. But, if $\Sigma_{SD}$ is a strong Maltsev condition for congruence meet semidistributivity, then $\mathcal{W}_l$ has $\Sigma_{SD}$-terms. Composing interpretations produces $\Sigma_N$-terms for $\mathcal{W}_l$, which is a contradiction. 
\end{proof}

\bibliographystyle{abbrv}
\bibliography{global}
\end{document}